\documentclass[11pt,a4paper]{article}
\usepackage[utf8]{inputenc}
\usepackage[english]{babel}
\usepackage{amsmath, amsthm, amssymb, amsfonts}
\usepackage{geometry}
\usepackage{hyperref}

\newtheorem{theorem}{Theorem}[section]

\newtheorem{proposition}[theorem]{Proposition}
\newtheorem{corollary}[theorem]{Corollary}
\newtheorem{defn}{Definition}
\newtheorem{obs}{Observation}

\usepackage{amsthm} 

\theoremstyle{definition} 

\title{Corrections to classical results on Independence and Covering numbers of the Splitting graph}

\author{J. Castro, J. Leaños and O. Rosario \\ \small Universidad Autónoma de Zacatecas \\ \small Universidad Autónoma de Guerrero 
\\ \small \texttt{castrosimonjair@gmail.com}}
\date{\today}

\begin{document}

		\maketitle
	
\begin{abstract}
The splitting graph $S(G)$ of a finite simple graph $G$ was introduced by Sampathkumar and Walikar in 1980~\cite{SW1980} and has been extensively studied in relation to graph invariants of $G$. In their original work, several formulas relating the independence number and the vertex cover number of $S(G)$ to the corresponding parameters of $G$ were stated and subsequently cited in the literature.  In this paper, we show that some of these classical equalities do not hold in general. We present explicit counterexamples disproving the published results concerning independence and vertex cover numbers of splitting graphs. Moreover, we establish the correct formulas and precisely characterize the cases in which the original statements are valid and those in which they fail. These results correct an error that has remained unnoticed for more than four decades and provide a clearer understanding of splitting graphs from the perspective of independence and vertex cover number.
\end{abstract}

\section{Introduction}
Throughout this paper, $G=(V,E)$ denotes a simple finite graph with vertex set $V$ and edge set $E$. Let $u$ and $v$
be distinct vertices of $V$. If $u$ and $v$ are adjacent, we simply write $uv\in E$. Similarly, we shall use 
$N_G(v)$ to denote the {\em neighborhood} of $v$ in $G$, that is, the set of all vertices of 
$G$ that are adjacent to $v$. The number $|N_G(v)|$ is the {\em degree} of $v$ and is denoted by $d_G(v)$. 
If $d_G(v)=1$, we will say that $v$ is a {\em pendant} vertex of $G$.  
 
The {\em splitting graph} $S(G)$ of $G$ is constructed by adding to $G$ a new vertex $v'$ for each $v\in V$ and connecting $v'$ to $u\in V$ if and only if $uv\in E$. We denote the vertex set (respectively, edge set) added to $G$ to form $S(G)$ as $V'$ (respectively, $E'$). Then the vertex set of $S(G)$ is the disjoint union of $V$ and $V'$, the edge set of $S(G)$ 
is the disjoint union of $E$ and $E'$, $|V|=|V'|$ and $2|E|=|E'|$. 

The notion of splitting graphs was introduced by E. Sampathkumar and H. B. Walikar in 1980~\cite{SW1980}, and since then, it has been the subject of extensive study. As we have mentioned, part of the initial research~\cite{SW1980} focused on the relationships between $S(G)$ and $G$ concerning graph invariants. The study of such relationships remains an important area of research. Some few examples of graph invariants that have been studied from this perspective include the energy, edge hub-integrity, Wiener index, Harary index, and chromatic number variants~\cite{indices,energy,hub,coloring}. Another common direction of research involves generalizing the concept in various ways. For instance, in~\cite{signed} splitting signed graphs were studied, in~\cite{m-shadow}, $m$-splitting graphs and $m$-shadow graphs were addresed, and in~\cite{line}, the line splitting graph of a graph was defined. The classification problem of splitting graphs based on their combinatorial properties is also of great interest~\cite{magic, signed, clasification}.     
  
Following~\cite{SW1980}, we use $\alpha_0,~\alpha_1, \beta_0$ and $\beta_1$  to denote  the vertex cover number, edge cover number, independence number, and matching number of a graph, respectively. An equivalent assertion of the following proposition can be found in~\cite{hub, SW1980}. 

\begin{proposition}\label{p:wrong}\textup{(Proposition~3 in~\cite{SW1980})}
For any graph $G$ with $n$ vertices, \\
(i) $\alpha_0(S(G))=n=\beta_0(S(G))$.\\
(ii) $\alpha_1(S(G))=2\alpha_1(G)$ and $\beta_1(S(G))=2\beta_1(G)$.
\end{proposition} 

While studying basic properties of splitting graphs, we observed discrepancies in the equalities stated in Proposition~\ref{p:wrong} $(i)$. A detailed analysis shows that these equalities are not valid in general. To the best of our knowledge, this issue has not been previously identified in the literature. The purpose of this paper is to establish the correct relationships between the independence and vertex cover numbers of a graph and those of its splitting graph, and to characterize precisely the cases in which the original statements hold and those in which they fail.

Our main results are stated in terms of the following parameter.

\begin{defn} Let $G=(V,E)$ be a finite simple graph of order $n\geq 2$.  We define
$$\beta_0^*(G):=\max \{|S|-|N(S)|~:~S \mbox{ is an independent set of } G\},$$
where $N(S)$ denotes the neighbourhood of $S$ in $G$.  
\end{defn}
  
We now present some properties of $\beta_0^*(G)$ that follow directly from its definition.

\begin{obs}\label{o:def} Let $G=(V,E)$ be a finite simple graph of order $n\geq 2$. The following hold:
\begin{itemize}
\item[$(i)$] Since the empty set is independent in $G$, then $\beta_0^*(G)\geq 0$.
\item[$(ii)$] If $E=\emptyset$, then $\beta_0^*(G)=n$; otherwise $\beta_0^*(G)\leq n-2$.    
\item[$(iii)$] $\beta_0^*(G)=0$ if and only if $|N(S)|\geq |S|$ for every independet set $S$ of $G$.
\item[$(iv)$] If $\beta_0(G)>n/2$, then $\beta_0^*(G)>0$. Moreover, there exist graphs with $\beta_0^*(G)>0$ while
$\beta_0(G)\leq n/2$.
\end{itemize} 
\end{obs}
    
 Since the independence and vertex cover numbers are additive over connected components, we restrict attention in this note to connected graphs.
    
\begin{theorem}\label{t:main} Let $G=(V,E)$ be a finite simple connected graph of order $n\geq 2$. Then
$\beta_0(S(G)) = n+\beta^*_0(G)$. 
\end{theorem} 
	
The following proposition is an immediate consequence of the well-known Gallai's identity~\cite{galai} 
and the definition of $S(G)$.

\begin{proposition}\label{p:gallai} If $G$ is a finite simple graph of order $n\geq 2$, then 
$\alpha_0(S(G))+\beta_0(S(G))=2n$.
\end{proposition} 

Combining  Theorem~\ref{t:main} and Proposition~\ref{p:gallai} we have the following.
	
\begin{corollary}\label{c:alpha} Let $G=(V,E)$ be a finite simple connected graph of order $n\geq 2$.  Then
$\alpha_0(S(G)) = n-\beta^*_0(G)$. 
\end{corollary} 

The following complete characterization of the graphs attaining equalities in Proposition~\ref{p:wrong} $(i)$ follows directly from  Observation~\ref{o:def}~($iii$), Theorem~\ref{t:main}, and Corollary~\ref{c:alpha}.

\begin{corollary}\label{c:right} Let $G=(V,E)$ be a finite simple connected graph of order $n\geq 2$.  Then
$\alpha_0(S(G))=n=\beta_0(S(G))$ if and only if  $|N(S)|\geq |S|$ for every independet set $S$ of $G$.
\end{corollary} 

Our next corollary follows readily from the final paragraph of the proof of Theorem~\ref{t:main}.

\begin{corollary}\label{c:sets} Let $G=(V,E)$ be a finite simple connected graph of order $n\geq 2$.  Then
$X$ is a maximum independent set in $S(G)$ if and only if $\beta^*_0(G)=|A|-|N(A)|$, where $A:=X\cap V$.
\end{corollary}

\section{Proof of Theorem~\ref{t:main}}
According to the statement of Theorem~\ref{t:main}, we can assume that $G=(V,E)$ is a finite simple connected graph of order 
$n=|V|\geq 2$. Then, $S(G)=(V\cup V', E\cup E')$ and $|V|=|V'|=n$, where $V'$ and $E'$ are the sets of added vertices and edges, respectively. We denote the vertex in $V'$ corresponding to $v \in V$ as $v'$  (and vice versa).
 
First we show that $\beta_0(S(G)) \geq n+\beta^*_0(G)$. From the definition of $S(G)$, we know that $V'$ is an independent set of $S(G)$ of order $n$. Then, we may assume that $\beta^*_0(G)>0$, as otherwise we are done, because 
$\beta_0(S(G)) \geq |V'|=n+0$. 

Let $S\subseteq V$ be an independent set of $G$ such that $\beta^*_0(G)=|S|-|N(S)|$. 
We observe that if $N':=\{v'\in V'~:~v\in N(S)\}$, then the set $X:=S\cup (V'\setminus N')$ is an independent set of $S(G)$.
Since $|N'|=|N(S)|$, then $|X|=|S|+n-|N(S)|=n+\beta^*_0(G)$, as required. 

We now show that $\beta_0(S(G)) \leq n+\beta^*_0(G)$. Let $X$ be a maximum independent set of $S(G)$; we must
prove $|X|\leq n+\beta^*_0(G)$. Define $A:= X \cap V$ and $B':= X \cap V'$. If $A=\emptyset$, then $X=B'\subseteq V'$ and 
$|X|\leq |V'|=n$, so the bound holds. Thus assume $A\neq \emptyset$ and let $N':=\{v'\in V'~:~v\in N(A)\}$. From the 
construction of $S(G)$ and the independence of $A\cup B'$ we have $B'\cap N'=\emptyset$, hence  
$|B'|\leq n - |N'|$. Since $|N'|=|N(A)|$, we have that $|X|=|A|+|B'|\leq |A|+n-|N(A)|\leq n+\beta^*_0(G)$, as required.

\section{Counterexamples} 
Since $G$ is connected, then $0\leq \beta^*_0(G)\leq n-2$ by Observation~\ref{o:def}~$(i)$-$(ii)$.
From this fact and Theorem~\ref{t:main}, it follows that 
$$n\leq \beta_0(S(G))\leq 2n-2.$$

\begin{corollary} \label{cor:any-n} Let $n\geq 2$ be an integer. For any $k\in \{n,\ldots , 2n-2\}$ there exists a connected graph  of order $n$ such that $\beta_0(S(G))=k$ (and hence $\alpha_0(S(G))=2n-k$).
\end{corollary}	
 
\begin{proof} Let $k\in \{n, \ldots , 2n-2\}$. By Theorem~\ref{t:main}, it suffices to exhibit a connected graph $G_k$ of order $n$ with $\beta_0^*(G_k)=k-n$. 

If $k=n$, take $G_n$ as the complete graph $K_n$. Then $\beta^*_0(G_n)=0=n-n$. For $k\in \{n+1, \ldots , 2n-2\}$, let $G_k$ be the graph obtained from $K_{2n-k-1}$ by attaching a set $P$ of $k-n+1$ pendant vertices to a single vertex $u$ of $K_{2n-k-1}$. It is not hard to see that $\beta^*_0(G_k)=|P|-|N(P)|=(k-n+1)-1=k-n$, as required. 
\end{proof}
		
These constructions explicitly show that the equalities in Proposition~\ref{p:wrong} $(i)$ do not hold in general, and that  
$\beta_0(S(G))$ attains every value in the permissible range.
	
\section*{Acknowledgements}

The first author (Jair Castro Simon, CVU: 672561) is a postdoctoral fellow at the Universidad Autónoma de Zacatecas “Francisco García Salinas”, supported by the Secretaría de Ciencia, Humanidades, Tecnología e Innovación (SECIHTI) under the program “Estancias Posdoctorales por México Convocatoria 2025” (Estancia Posdoctoral Inicial), for the project: “Estudio estructural del diferencial de gráficas bajo la acción de operadores combinatorios clásicos” (No. 12977270).


\begin{thebibliography}{9}




\bibitem{indices} Francis Joseph H. Campe\~na, Ma. Christine G. Egan and John Rafael M. Antalan,
On the Weiner and Harary index of splitting graphs, {\em European Journal of pure and  applied mathematics},
Vol. 15, No. 2, 2022, 602--619.


\bibitem{energy}
Zheng-Qing Chu, Saima Nazeer, Tariq Javed Zia, Imran Ahmed and Sana Shahid, Some new results on various graph energies of the splitting graph. {\em Journal of Chemestry}, (2019). https://doi.org/10.1155/2019/7214047

\bibitem{galai}
T. Gallai, \"Uber extreme Punkt-und Kantenmengen. {\em Ann. Univ. Sci. Budapest, E\"otv\"os Sect. Math.}, {\bf 2} (1959), 133--138.

\bibitem{hub} Sultan Senan Mahde and Veena Mathad,
Some results on the edge hub-integrity of graphs, {\em Asia Pacific Journal of Mathematics},
 Vol. 3, No. 2 (2016), 173--185

\bibitem{m-shadow} T. Manjula and R. Rajeswari, Dominator chromatic number of $m$-splitting graph and 
$m$-shadow graph of path graph, {\em International Journal of Biomedical Engineering and Technology},
Vol. 27, No. 1--2, (2018), 100--113. https://doi.org/10.1504/IJBET.2018.093089

\bibitem{magic} S. Muthukkumar, K. Rajendran,
Generation of anti-magic graphs, {\em Int. J. Anal. Appl.} (2025), 23--30. https://doi.org/10.28924/2291-8639-23-2025-30 

\bibitem{coloring} D. Muthuramakrishnan and G. Jayaraman,
Total coloring of splitting graph of path, cycle and star graphs, {\em International Journal of Mathematics and its Applications}, 6 (1--D) (2018), 659 -- 664.

\bibitem{line} V. R. Kulli and M. S. Biradar,
The line splitting graph of a graph, {\em Acta Ciencia Indica},
 Vol. XXVIII M, No. 3, (2002) (2016), 317--322. 

\bibitem{signed} Sandeep Kumar and Deepa Sinha,
Spectral analysis of splitting signed graph, {\em European Journal of pure and  applied mathematics}, 
Vol. 17, No. 1, 2024, 504--518.

\bibitem{clasification} R. Ponraj and S. Sathish Narayanan,
Difference cordiality of come derived graphs, {\em International J.Math. Combin.} Vol. 4 (2013), 37--48.
 
\bibitem{SW1980} E. Sampathkumar and H. B. Walikar, On the splitting graph of a graph, {\em J. Karnatak Univ. Sci.}, 25 (1980) 13--16.
		
	\end{thebibliography}
\end{document}